\documentclass{amsart}

\usepackage{amssymb} \usepackage{amsfonts} \usepackage{amsmath}
\usepackage{amsthm} 
\usepackage{epsfig} 
\usepackage{color}
\usepackage{amscd}
\usepackage{comment}
\usepackage[all]{xy}
\usepackage[english]{babel}
\usepackage{hyperref}
\usepackage{amssymb, amsfonts, amsmath}
\usepackage{pgfplots}
\usepackage{comment} 
\usepackage{mathtools}
\usepackage{wrapfig}
\usepackage{tikz}
\usepackage{comment}
\usepackage{mathrsfs}
\usepackage{tabularx, hyperref}
\usepackage{amssymb} \usepackage{amsfonts} \usepackage{amsmath}
\usepackage{amsthm} \usepackage{epsfig, subfig}
\usepackage{ amscd, amsxtra, latexsym}
\usepackage{epsfig,  graphicx, psfrag}
\usepackage[all]{xy}
\usepackage{caption}
\usepackage{enumerate}
\usepackage{color}

\newtheorem{lemma}{Lemma}[section]
\newtheorem{thm}[lemma]{Theorem}
\newtheorem{prop}[lemma]{Proposition}
\newtheorem{cor}[lemma]{Corollary}

\newtheorem*{prop*}{Proposition}
\newtheorem{prop_intro}{Proposition}

\newtheorem{thm_intro}[prop_intro]{Theorem}

\theoremstyle{definition}
\newtheorem{defn_intro}[prop_intro]{Definition}
\newtheorem{defn}[lemma]{Definition}

\theoremstyle{definition}

\newtheoremstyle{citing}
{3pt}
{3pt}
{\itshape}
{}
{\bfseries}
{.}
{.5em}
{\thmnote{#3}}
\theoremstyle{citing}
\newtheorem*{varthm}{}

\newcommand{\matN}{\ensuremath {\mathbb{N}}}
\newcommand{\matR} {\ensuremath {\mathbb{R}}}
\newcommand{\matQ} {\ensuremath {\mathbb{Q}}}
\newcommand{\matZ} {\ensuremath {\mathbb{Z}}}
\newcommand{\matC} {\ensuremath {\mathbb{C}}}

\newcommand{\matH} {\ensuremath {\mathbb{H}}}

\newcommand{\calM} {\ensuremath {\mathcal{M}}}

\newcommand{\inte}[1]{{\rm int}(#1)}

\newcommand{\nota} [1] {\caption{\footnotesize{#1}}}

\newcommand{\eq}{{\rm eq}}

\newcommand{\alt}{{\rm alt}}
\newcommand{\inc}{{\rm inc}}
\newcommand{\thi}{{\rm th}}

\newcommand{\haar}{\textrm{Haar}}

\newcommand{\str} {\ensuremath {{\rm str}}}
\newcommand{\strtil} {\ensuremath {\widetilde{\rm str}}}

\DeclareMathOperator{\reg}{{Reg}}
\DeclareMathOperator{\aut}{{Aut}}

\newcommand{\vare}{\varepsilon}

\DeclareMathOperator{\isom}{Isom}
\DeclareMathOperator{\vol}{Vol}

\newcommand{\commento}[1]{}

\author{Roberto Frigerio}
\address{Dipartimento di Matematica, Largo Pontecorvo 5, 56127 Pisa, Italy}
\email{roberto.frigerio@unipi.it}

\author{Ennio Grammatica}
\address{Ecole Normale Sup\'erieure, 
4 Av. des Sciences,  91190 Gif-Sur-Yvette, France}
\email{ennio.grammagica@ens-paris-saclay.fr}

\author{Bruno Martelli}
\address{Dipartimento di Matematica, Largo Pontecorvo 5, 56127 Pisa, Italy}
\email{bruno.martelli@unipi.it}

\title[Efficient cycles of cusped $3$-manifolds]{Efficient cycles \\ of hyperbolic manifolds}


\begin{document}

\begin{abstract}
Let $N$ be a complete finite-volume hyperbolic $n$-manifold. An efficient cycle for $N$ is the limit (in an appropriate measure space) of a sequence of fundamental cycles whose $\ell^1$-norm converges to the simplicial volume of $N$. 
Gromov and Thurston's smearing construction  
exhibits an explicit efficient cycle, and Jungreis and Kuessner proved that, in dimension $n\geq 3$, such cycle actually is the unique efficient cycle for a huge class of finite volume hyperbolic manifolds, including all the closed ones.
In this paper we prove that, for $n\geq 3$, the class of finite-volume hyperbolic manifolds
for which the uniqueness of the efficient cycle does not hold is exactly the commensurability class of
 the figure-8 knot complement (or, equivalently, of the Gieseking manifold). 
\end{abstract}

\maketitle

\section*{Introduction}

The simplicial volume is a homotopy invariant of manifolds introduced by Gromov in his pioneering paper~\cite{Gromov}. If $N$ is a compact connected oriented $n$-manifold
(possibly with boundary) the simplicial volume $\|N\|$ of $N$ is defined by
$$
\|N\|=\inf \left\{\sum_{i=1}^k |a_i|\, :\, \Big[\sum_{i=1}^k a_i\sigma_i\Big]=[N]\in H_n(N,\partial N)\right\}\ ,
$$
where $[N]$ denotes the real fundamental class of $N$, and $H_n(N,\partial N)$ denotes the relative singular homology module of the pair
$(N,\partial N)$ with real coefficients.

Computing the simplicial volume is usually a very difficult task. Many vanishing theorems are available by now, but positive exact values of the simplicial volume are known
only for a few classes of manifolds, such as complete finite-volume hyperbolic manifolds~\cite{Gromov, Thurston}, closed manifolds isometrically covered by the product of two copies
of the hyperbolic plane~\cite{Bucher3}, some $3$-manifolds with higher genus boundary~\cite{BFP} and special families of $4$-manifolds~\cite{HL}.
Even when the simplicial volume of a manifold $N$ is known, characterizing (or, at least, exhibiting) \emph{almost minimal} fundamental cycles (i.e.~fundamental cycles
whose norm is close to $\|N\|$) may be surprisingly difficult. For example, it is known that the simplicial volume of any closed simply connected manifold $N$ vanishes, but 
there is no recipe, in general, which describes fundamental cycles of $N$ of arbitrarily small norm; in a similar spirit, even if the value of the simplicial volume of the product 
$\Sigma\times \Sigma'$ of two hyperbolic surfaces has been computed in~\cite{Bucher3}, exhibiting a sequence of fundamental cycles whose norm approximates $\|\Sigma\times\Sigma'\|$ 
seems very challenging~\cite{Marasco}.

The situation is better understood for hyperbolic manifolds: the computation by Gromov and Thurston of the simplicial volume of such manifolds explicitly constructs 
almost minimal cycles via 
an averaging operator called \emph{smearing}~\cite{Thurston}.
A natural question is to
which extent this construction is unique, i.e., whether there exist sequences
of almost minimal fundamental cycles which do not come from smearing: this problem has been partially addressed by Jungreis and Kuessner in~\cite{Jungreis, Kue:efficient}. 

In this paper we improve their results by showing that, in dimension $n\geq 3$,  the unique hyperbolic manifolds 
admitting ``exotic''  almost minimal fundamental cycles are those which are commensurable with the Gieseking manifold (it is known that hyperbolic surfaces admit many 
almost minimal efficient cycles which do not come from smearing, see e.g.~\cite[Remark at page 647]{Jungreis}). 

In order to state more precisely our results, let us introduce some notation. Let $N$ be a complete finite-volume hyperbolic $n$-manifold. 
If $N$ is closed, we denote by $\|N\|$ its simplicial volume. If $N$ is
 non-compact, it is the internal part of a compact manifold with boundary $\overline{N}$, and for the sake of simplicity
 we still denote by $\|N\|$ the simplicial volume of $(\overline{N},\partial \overline{N})$. In fact, by replacing finite chains with locally finite ones, the definition of simplicial volume may be
 extended to open manifolds, and for finite-volume hyperbolic manifolds this notion of simplicial volume coincides with the simplicial volume of the compactification 
 (see e.g.~\cite{KimKue}). In order to better compare our results with Kuessner's we prefer to work with the relative simplicial volume of compact manifolds with boundary rather than the simplicial volume of open manifolds, even if our proofs can be easily adapted also to the latter framework.

Let $c_i$, $i\in\mathbb{N}$, be a sequence of (relative) fundamental cycles such that 
$$\lim_{i\to +\infty} \|c_i\|=\|N\|.$$ 
Any possible limit $\mu$ of such a sequence naturally sits
in the space $\mathcal{M}(\overline{S}_n^*(N))$ of signed measures on the space of (nondegenerate and possibly ideal) geodesic simplices in $N$,
and will be called an \emph{efficient cycle} for $N$: thus, an efficient cycle is a measure rather than a classical chain (we refer the reader to Section~\ref{preliminaries}
for more details). In fact, it is not difficult to prove that an efficient cycle $\mu$ is supported on the subspace $\reg(N)$ of \emph{regular ideal} simplices, which may be identified with
$\Gamma\backslash \isom(\mathbb{H}^n)$, where $\Gamma$ is the subgroup of $\isom(\mathbb{H}^n)$ such that $N=\Gamma\backslash \mathbb{H}^n$
(see Lemma~\ref{support}). The Haar measure on $\isom(\mathbb{H}^n)$ may then be exploited to define an uniformly distributed measure $\mu_\eq$ on $\reg(N)$, and Gromov and Thurston's smearing procedure constructs
sequences of fundamental cycles converging exactly to a suitable multiple of $\mu_\eq$.

Jungreis and Kuessner   provided a complete characterization of efficient cycles of finite-volume hyerbolic $n$-manifolds, $n\geq 3$, except that for the
exceptional family consisting of the so-called \emph{Gieseking-like} manifolds. 

\begin{defn_intro}\label{Gieseking:defn}
Let $N=\Gamma\backslash \matH^3$ be a cusped hyperbolic $3$-manifold. Let us fix an identification between $\partial \matH^3$ and the space $\matC\cup \{\infty\}$, and
let $\mathcal{P}=\matQ(e^{i\pi/3})\cup \{\infty\}\subseteq \partial \matH^3$. Then $N$ is \emph{Gieseking-like} if there exists a conjugate $\Gamma'$ of $\Gamma$ in
$\isom(\matH^3)$ such that $\mathcal{P}$ is contained in the set of parabolic fixed points of $\Gamma'$.
\end{defn_intro}

The well-known Gieseking manifold is indeed Gieseking-like. Moreover, being Gieseking-like is invariant with respect to commensurability, hence all hyperbolic $3$-manifolds
which are commensurable with the Gieseking manifold (like, for example, the figure-8 knot complement) are Gieseking-like. It is still unknown whether the class of Gieseking-like manifolds coincides with
the commensurability class of the Gieseking manifold, or it is strictly larger (see~\cite{LongReid}).   

Let $v_n$ be the volume of a regular ideal simplex in hyperbolic space $\mathbb{H}^n$.
We are now ready to state Jungreis' and Kuessner's results:

\begin{thm_intro}[\cite{Jungreis}]\label{Jungreis:thm}
Let $N$ be a closed orientable $n$-hyperbolic manifold with $n\geq 3$. Then $N$ admits a unique efficient cycle, which is given by the measure
$$
\frac{1}{2 v_n} \cdot \mu_\eq\ .
$$
\end{thm_intro}

\begin{thm_intro}[\cite{Kue:efficient}]\label{kuessner:thm}
Let $N$ be a complete finite-volume $n$-hyperbolic manifold, $n\geq 3$, and suppose that $N$ is not Gieseking-like (this condition is automatically satisfied
if $n\geq 4$). Then every efficient cycle of $N$ is a multiple of $\mu_\eq$.
\end{thm_intro}

In fact, Kuessner 
 states in~\cite[Theorem 4.5]{Kue:efficient} that any efficient cycle is a non-vanishing multiple of $\mu_\eq$, without explicitly computing the proportionality coefficient 
 $1/(2v_n)$ appearing in Jungreis'  Theorem. When $N$ is non-compact,  the  space of straight simplices in $N$ is non-compact, which introduces some issues when dealing with the weak-* convergence of measures (namely,  by passing to the limit there could be some loss of mass).

 Let us say that a measure in $\mathcal{M}(\reg(N))$ is \emph{equidistributed} if it is a multiple of $\mu_\eq$.
The main results of this paper strengthen and clarify Kuessner's result in three directions: 
\begin{enumerate}
\item
We prove
that indeed the total variation of any efficient cycle of a cusped hyperbolic manifold is equal to its simplicial volume, thus 
showing that, in the non-Gieseking like case, 
also for cusped manifolds
the proportionality coefficient between any efficient cycle and $\mu_\eq$ is equal to
 $1/(2v_n)$, as  in Jungreis'  Theorem;
 \item
we show that if a cusped $3$-manifold $N$ admits non-equidistributed efficient cycles,
then it is commensurable with the Gieseking manifold (a condition which is potentially stronger than being Gieseking-like); 
\item moreover, for any such manifold we exhibit non-equidistributed efficient cycles, thus obtaining a complete characterization of hyperbolic manifolds with non-unique efficient cycles.
\end{enumerate}

Let us state more precisely our results:

\begin{thm_intro}[No loss of mass]\label{totalvariation}
Let $N$ be a complete finite-volume hyperbolic $n$-manifold, $n\geq 3$, and let $\mu$ be an efficient cycle for $N$.
Then $\|\mu\|=\|N\|$.
\end{thm_intro}

\begin{thm_intro}\label{characterization}
Let $N$ be a complete finite-volume hyperbolic manifold. Then $N$ admits non-equidistributed efficient cycles if and only if it is commensurable with the Gieseking manifold.
\end{thm_intro}

Putting together Theorems~\ref{totalvariation} and~\ref{characterization} we can then deduce the following:

\begin{thm_intro}\label{summarize}
Let $N$ be a complete finite-volume hyperbolic $n$-manifold with $n\geq 3$. The following hold:
\begin{enumerate}
\item
If $N$ is not commensurable with the Gieseking manifold and $c_i$, $i\in\mathbb{N}$ is any minimizing sequence for $N$, then
$$
\lim_{i\to +\infty} c_i=\frac{1}{2v_n} \mu_\eq\ ;
$$
\item If $N$ is commensurable with the Gieseking manifold, then $N$ admits non-equidistributed efficient cycles.
\end{enumerate}
\end{thm_intro}

We can be more precise. If $N$ is commensurable with the Gieseking manifold, then a finite cover $M$ of $N$ admits a decomposition $T$ into regular ideal tetrahedra. 
The triangulation $T$ induces a measure cycle $\mu_T\in \mathcal{M}(\reg(M))$ which is a finite sum of atomic measures 
supported
on the regular ideal tetrahedra appearing in $T$ (see Subsection~\ref{mu:subsec} for the precise definition of $\mu_T$). We then have the following:

\begin{thm_intro}\label{explicit:thm}
Let $M$ be a complete finite-volume $3$-manifold admitting a decomposition $T$ into regular ideal tetrahedra. Then $\mu_T$ is an efficient cycle for $M$.
\end{thm_intro}
 
 From the non-equidistributed efficient cycle $\mu_T$ for $M$ one can then easily construct a non-equidistributed efficient cycle for $N$.
 The proof of Theorem~\ref{explicit:thm} exploits a construction described in~\cite{FFM}, which allows to replace an ideal triangulation  $T$ of a cusped manifold with a classical triangulation
 of its compactification  in a very controlled way. By applying this procedure to a suitably chosen tower of coverings of $M$ and pushing-forward the resulting classical triangulations to $M$ we obtain a minimizing sequence whose limit is equal to $\mu_T$.

\subsection*{Plan of the paper}
In Section~\ref{preliminaries} we recall the definition of simplicial volume, of minimizing sequence, of efficient cycle and of equidistributed efficient cycle. To this aim we also introduce the  measure spaces we will exploit throughout the paper. Section~\ref{properties} is devoted to the proof of some fundamental properties of efficient cycles, including Theorem~\ref{totalvariation}.
In Section~\ref{unique:sec} we prove that, if a complete finite-volume hyperbolic $n$-manifold, $n\geq 3$, admits a non-equidistributed efficient cycle, then
it is necessarily commensurable with the Gieseking manifold, while Section~\ref{nonunique:sec} is devoted to the construction of non-equidistributed efficient cycles for
manifolds which are commensurable with the Gieseking manifold.

\section{Preliminaries}\label{preliminaries}

\subsection{Simplicial volume}
Let $X$ be a topological space. For every $k\in\mathbb{N}$, we denote by $\hat{S}_k(X)$ the set of singular $k$-simplices with values in $X$, and by
$C_k(X)$ the chain module of singular $k$-chains with \emph{real} coefficients, i.e.~the real vector space with free basis $\hat{S}_k(X)$.
If  $Y\subseteq X$, we denote by $C_*(X,Y)$ the chain complex of relative singular cochains with real coefficients, and by $H_*(X,Y)$ the corresponding homology module. 
We endow $C_*(X)$ with an $\ell^1$-norm $\|\cdot\|$ defined by
$$
\left\| \sum_{i=1} a_i\sigma_i\right\|=\sum_{i=1}^k |a_i|\ .
$$
This norm descends to a norm on $C_*(X,Y)$ and, by taking the infimum over representatives, to a seminorm on $H_*(X,Y)$, still denoted by $\|\cdot \|$. 

If $N$ is a compact oriented $n$-dimensional manifold (possibly with boundary), then the singular homology module with integral coefficients $H_n(N,\partial N;\mathbb{Z})\cong \mathbb{Z}$
is generated by the \emph{integral fundamental class} $[N]_\mathbb{Z}\in H_n(N,\partial N,\mathbb{Z})$. Under the change of coefficient homomorphism induced by
the inclusion $\mathbb{Z}\hookrightarrow \mathbb{R}$,
the class $[N]_\mathbb{Z}$ is sent to the \emph{real fundamental class} $[N]\in H_n(N,\partial N;\mathbb{R})$. 

\begin{defn}[{\cite{Gromov}}]
The simplicial volume of $N$ is 
$$
\|N\|=\| [N]\|\ .
$$
\end{defn}

Henceforth,  all the homology modules will be understood with real coefficients, and the coefficients will be omitted from our notation.

\subsection{Straight chains on hyperbolic manifolds}
Let $N=\Gamma\backslash \matH^n$ be a cusped oriented hyperbolic $n$-manifold, where $\Gamma$ is a discrete subgroup of $\isom^+(\matH^n)$. For every $k\in\mathbb{N}$,
if $\sigma\colon \Delta^k\to \matH^n$ is a singular simplex, we denote by $\strtil_k(\sigma)$ the \emph{straightening} of $\sigma$, i.e.~the singular simplex
obtained by suitably  parametrizing the convex hull of the vertices of $\sigma$ (see e.g.~\cite[Section 8.7]{frigerio:book} or~\cite[Chapter III.13]{martelli:book}). We denote by $S_k(\matH^n)\subseteq \hat{S}_k(\matH^n)$ the image 
of $\strtil_k$, i.e.~the subset of \emph{straight} hyperbolic $k$-simplices, and observe that there is a natural identification $S_k(\matH^n)=(\matH^n)^{k+1}$ sending a straight simplex to
the (ordered) set of its vertices, which will be understood henceforth. With a slight abuse, we still denote by $\strtil_k\colon C_k(\matH^n)\to C_k(\matH^n)$ the $\mathbb{R}$-linear extension of $\strtil_k$ to the space of singular chains,
and recall that $\strtil_*\colon C_*(\matH^n)\to C_*(\matH^n)$ is in fact a chain map (see e.g.~\cite[Proposition 8.11]{frigerio:book}). 

Being $\isom(\matH^n)$-equivariant (and equivariantly homotopic to the identity),
the map $\strtil_k$ is in particular $\Gamma$-equivariant, and  induces a well-defined chain map $\str_*\colon C_*(N)\to C_*(N)$, which is chain homotopic to the identity. 
We denote by
$S_k(N)$ the image of $\hat{S}_k(N)$ via $\str_k$, i.e.~the set of straight simplices in $N$, and we observe that there is a natural identification
$$S_k(N)=\Gamma\backslash S_k(\matH^n)=\Gamma\backslash (\matH^n)^{k+1}\ .$$
A chain is called \emph{straight} if it is supported on straight simplices or, equivalently, if it lies in the image of the chain map $\str_*$ (or $\strtil_*$).

We denote by $SC_*(\matH^n)\subseteq C_*(\matH^n)$ (resp.~$SC_*(N)\subseteq C_*(N)$) the complex of straight chains in $\matH^n$ (resp.~in $N$). By construction,  under the above identification
between the set of straight $k$-simplices and $(\matH^n)^{k+1}$ (resp. $\Gamma\backslash (\matH^n)^{k+1}$), 
the complex $SC_*(\matH^n)$  (resp.~$SC_*(N)$) is identified with the free vector space with basis  $(\matH^n)^{*+1}$ (resp.~$\Gamma\backslash (\matH^n)^{*+1}$), with boundary operators which linearly extend the maps
$$
\partial_k (v_0,\dots,v_k)=\sum_{i=0}^k (-1)^i (v_0,\dots,\widehat{v}_i,\dots,v_k)
$$
(resp. $\partial_k [(v_0,\dots,v_k)]=\sum_{i=0}^k (-1)^i [(v_0,\dots,\widehat{v}_i,\dots,v_k)]$). 

If $\sigma=(v_0,\dots,v_k)\in(\matH^n)^{k+1}$ is a straight simplex, the alternating chain associated to $\sigma$ is defined by
$$
\alt_k(\sigma)=\frac{1}{k!} \sum_{\tau\in\mathfrak{S}_{k+1}} \varepsilon(\tau)\cdot (v_{\tau(0)},\dots, v_{\tau(k)})\ ,
$$
where $\mathfrak{S}_{k+1}$ is the group of the permutations of the set $\{0,\dots,k\}$, and $\varepsilon(\tau)=\pm 1$ is the sign of $\tau$ for every $\tau \in \mathfrak{S}_{k+1}$.
The maps $\alt_k$ linearly extend to a chain map $\alt_*\colon SC_*(\matH^n)\to SC_*(\matH^n)$ which is $\Gamma$-equivariant, and $\Gamma$-equivariantly homotopic
to the identity (see e.g.~\cite[Appendix A]{Fujiwara_Manning}). In particular, $\alt_*$ induces a well-defined chain map
$\alt_*\colon SC_*(N)\to SC_*(N)$, which is homotopic to the identity. A chain in $SC_*(\matH^n)$ (or in $SC_*(N)$) is \emph{alternating}
if it lies in the image of $\alt_*$.

\subsection{Thick-thin decomposition of hyperbolic manifolds}
For every $\varepsilon>0$ we denote by $N_\varepsilon$ the $\vare$-thick part of $N$, i.e.~the set of points of $N$
whose injectivity radius is not smaller than $\vare$. We will always choose $\varepsilon>0$ small enough so that $N_\vare$ is a compact submanifold with boundary of $N$, 
obtained from $N$ by removing open neighbourhoods of its cusps. We denote by $\overline{N}$ the natural compactification of $N$, which is diffeomorphic to $N_\vare$. 
The inclusion $(N_\vare,\partial N_\vare)\to (N,N\setminus \inte{N_\vare})$ and the obvious deformation retraction $r\colon  (N,N\setminus \inte{N_\vare})\to (N_\vare,\partial N_\vare)$ are one the homotopy inverse of the other, and they induce norm non-increasing maps in homology. Therefore, in order to compute the simplicial volume of $\overline{N}$ we may consider
relative fundamental cycles in $C_n(N,N\setminus \inte{N_\vare})$. Morevover,
the complement in $\matH^n$ of the preimage of $N_\vare$ under the covering projection is an equivariant family of disjoint open horoballs. Since horoballs are convex in $\matH^n$,
the straightening operator induces a well-defined chain map on the relative chain complex $C_*(N,N\setminus \inte{N_\vare})$. 
Finally, since both the straightening  and the alternating operators are obviously norm non-increasing (and they induce the identity also on relative homology), in order to  compute the simplicial volume of $N$
it is not restrictive to consider only straight and alternating relative cycles in $C_*(N,N\setminus \inte{N_\vare})$.

\subsection{Minimizing sequences and efficient cycles}
We say that a sequence $c_i\in C_n(N)$ of chains is a \emph{minimizing sequence} if the following conditions hold:
\begin{enumerate}
\item Each $c_i$ is straight and alternating;
\item For every sufficiently large $i\in\mathbb{N}$, the chain $c_i$ is a relative cycle in $C_n(N_{2^{-i}}, N\setminus \inte{N_{2^{-i}}})$;
\item Under the identification $H_n(N_{2^{-i}}, N\setminus \inte{N_{2^{-i}}})\cong 
H_n (\overline{N},\partial \overline{N})$ described above, the relative cycle $c_i$ represents
the fundamental class of $(\overline{N},\partial \overline{N})$;
\item $\|c_i\|\leq \|N\|+2^{-i}$ for all sufficiently large $i\in\mathbb{N}$. 
\end{enumerate}

Of course, in the definition of minimizing sequence the values $2^{-i}$ may be replaced by any infinitesimal sequence $\eta_i$; we decided to choose this
specific sequence just to simplify the notation.

We now introduce the measure spaces we are interested in. Recall that $S_n(N)=\Gamma\backslash ({\matH^n})^{n+1}$ is the space of straight simplices with values in $N$. 
Of course, this space does not contain any ideal simplex,  hence we need to enlarge it in order to construct a measure space which could support possible limits of minimizing sequences (recall from the introduction that efficient cycles are supported on regular ideal simplices). 
The natural space to look at is then  $\overline{S}_n(N)=\Gamma\backslash (\overline{\matH^n})^{n+1}$ but, unfortunately, the action
of $\Gamma$ on $\overline{S}_n(\matH^n)=(\overline{\matH^n})^{n+1}$ has not closed orbits, so that the quotient space is not Hausdorff. In order to avoid this inconvenience, and for other
later purposes,
we introduce the following:

\begin{defn}
A simplex in $\overline{S}_n(\mathbb{H}^n)$ is \emph{degenerate} if its vertices (hence, its image) lie on (the closure at infinity of) a hyperplane of $\mathbb{H}^n$ or,
equivalently, if its image has volume equal to $0$. A simplex in $ \overline{S}_n(N)=\Gamma\backslash \overline{S}_n(\mathbb{H}^n)$ is degenerate if it is the image of
a degenerate simplex in $\overline{S}_n(\mathbb{H}^n)$. 

We denote by $\overline{S}^*_n(\mathbb{H}^n)$ (resp.~$\overline{S}^*_n(N)$) the set of nondegenerate simplices in 
$\overline{S}_n(\mathbb{H}^n)$ (resp.~$\overline{S}_n(N)$) .
\end{defn}


It is not difficult to show that, when endowed with the quotient topology, the space $\overline{S}_n^*(N)$ is Hausdorff and locally compact (see e.g.~\cite[Lemma 2.6]{Kue:efficient} for a similar result). 
We denote by $\mathcal{M}(\overline{S}_n^*(N))$ the space of signed regular measures on $\overline{S}_n^*(N)$.
If $\sigma\colon \Delta_n\to N$ is a straight simplex,
then we denote by $\delta_\sigma \in \mathcal{M}(\overline{S}_n^*(N))$ the atomic measure concentrated on  $\sigma$.
The map $\sigma\mapsto \delta_\sigma$ linearly extends to a map
$$
\Theta\colon SC_n(N) \to \mathcal{M}(\overline{S}_n^*(N))\ .
$$


We are now ready to define the notion of efficient cycle for complete finite-volume hyperbolic manifolds:

\begin{defn}\label{efficient:defn}
A measure $\mu\in \mathcal{M}(\overline{S}_n^*(N))$ is an \emph{efficient cycle} for $N$ if 
there exists a minimizing sequence $c_i$, $i\in\matN$ such that 
$$
\mu =\lim_{i\to +\infty} \Theta (c_i)\ ,
$$
where the limit is taken with respect to the weak-* topology on $ \mathcal{M}(\overline{S}_n^*(N))$.
\end{defn}

\subsection{Equidistributed measure cycles}
As explained in the introduction, we are going to prove that, if $N$ is \emph{not} commensurable with the Gieseking manifold, then there exists a unique efficient cycle for $N$, which is concentrated on (classes of) regular ideal simplices,
and is moreover equidistributed on such simplices. Let us formally describe what we mean by equidistributed measure on (classes of) regular ideal simplices. 

We define 
$$\reg({\matH^n})=\{(v_0,\dots,v_n)\in(\partial \matH^n)^{n+1}\, |\,  v_0,\dots,v_{n}\ \text{span a regular ideal  simplex}\}$$
and we denote by $\reg^+({\matH^n})$ (resp.~$\reg^-(\partial{\matH^n})$) the subset of $\reg({\matH^n})$ corresponding 
to positively oriented (resp.~negatively oriented) simplices. 
We then set 
$$
\reg^\pm(N)=\Gamma\backslash \reg^\pm ({\matH^n})\subseteq \overline{S}_n^*(N)\ .
$$
Since $N$ is oriented, elements of $\Gamma$ are orientation-preserving, hence the sets $\reg^+(N)$ are $\reg^-(N)$ are disjoint. 

Let $\Delta_0=(v_0,\dots,v_n)\in(\partial \matH^n)^{n+1}$ be the (ordered) $(n+1)$-tuple of vertices of a fixed positively oriented regular ideal hyperbolic simplex. We then have bijections
$$
\isom^\pm (\matH^n)\to \reg^\pm({\matH^n})\, ,\qquad g\mapsto g\cdot \Delta_0=(g(v_0),\dots,g(v_n))\ ,
$$
which induce bijections
$$
\Gamma\backslash \isom^\pm(\matH^n)\to \reg^\pm (N)\ .
$$
We denote by the symbol $\haar$ the Haar measure on $\isom (\matH^n)$, normalized in such a way that, for every measurable subset $\Omega\subseteq \matH^n$ and any $x_0\in\matH^n$, 
$$
\haar \{ g\in \isom(\matH^n)\, |\, g(x_0)\in\Omega\}= \vol(\Omega)\ .
$$
Being bi-invariant, the Haar measure induces well-defined finite  measures 
$\haar_\pm$ on $\Gamma/\isom^\pm(\matH^n)$, hence on  $\reg^\pm(N)$ via the above identifications. 
We finally set 
$$
\mu_\eq=\haar_+-\haar_-\ \in\ \mathcal{M}(\reg(N))\ \subseteq \ \mathcal{M}(\overline{S}_n^*(N))\ ,
$$
where the subscript ``eq'' stands for ``equidistributed''.
Using again the bi-invariance of $\haar$ one can easily
check that the definition of $\haar_\pm$ (hence of $\mu_\eq$) on $\reg({\matH^n})$ does not depend on the choice of $\Delta_0$.

\section{Some properties of efficient cycles}\label{properties}
For every $\varepsilon>0$ we denote by $\omega_\varepsilon\colon C_n(N, N\setminus \inte{N_\vare})$ the restriction of the volume cochain to $N_\vare$, i.e.~the cochain such that
$$
\omega_\vare (c) =\int_c d\vol_\vare\ ,
$$
where $d\vol_\vare$ is the (discontinuous) $n$-form that coincides with the hyperbolic volume form on $N_\vare$ and is equal to $0$ on $N\setminus N_\vare$ (for our purposes, it is sufficient to define $\omega_\vare$ on straight chains, which of course are $C^1$, so the the integral above makes sense). If $c\in C_n(N)$
is a straight relative fundamental cycle for $(N, N\setminus \inte{N_\vare})$, then we have
$$
\omega_\vare (c)=\vol (N_\vare)\ .
$$
If $\sigma$ is a straight simplex with values in $N$, then it is immediate to check that $\omega_\vare(\sigma)=\pm \vol (\widetilde{\sigma}\cap \widetilde{N}_\vare)$,
where $\widetilde{\sigma}$ is a lift of $\sigma$ to $\mathbb{H}^n$, the space $\widetilde{N}_\vare$ is the preimage of $N_\vare$ in $\mathbb{H}^n$,
and the sign is positive (resp.~negative) if $\sigma$ is positively oriented (resp~negatively oriented). 

The following lemma shows that, in a minimizing sequence, the orientation of simplices has to be coherent with the sign of their coefficients, at least asymptotically.

\begin{lemma}\label{nonegative:lemma}
Let $c_i$, $i\in\mathbb{N}$ be a minimizing sequence, and for every $i\in\mathbb{N}$ let
$$
c_i=\sum_{k=1}^{n_i} a_{i,k} \sigma_{i,k}\ . 
$$ 
For every $i,k$, set $b_{i,k}=a_{i,k}$ if $a_{i,k}>0$ and $\sigma_{i,k}$ is not positively oriented or $a_{i,k}<0$ and $\sigma_{i,k}$
is not negatively oriented, and $b_{i,k}=0$ otherwise. 
If $c_i'=\sum_{k=1}^{n_i} b_{i,k}\sigma_{i,k}$, then
$$
\lim_{i\to +\infty} \|c'_i\|=0\ .
$$
\end{lemma}
\begin{proof}
By definition of minimizing sequence we have
$$
\lim_{i\to +\infty} \omega_{2^{-i}}(c_i)=\lim_{i\to +\infty} \vol (N_{2^{-i}})=\vol (N)\ ,
$$
hence
\begin{equation}\label{nonegative}
\lim_{i\to +\infty} \frac{\omega_{2^{-i}}(c_i)}{v_n}=\frac{\vol (N)}{v_n}=\|N\|=\lim_{i\to +\infty} \|c_i\|\ .
\end{equation}
Since the hyperbolic volume of any straight simplex is at most $v_n$, we have
$$
\omega_{2^{-i}} (c_i-c'_i)\leq \|c_i-c'_i\|\cdot v_n\ ,
$$
while our definition of $c_i'$ readily implies that
$
\omega_{2^{-i}}(c_i')\leq 0
$.
Therefore, 
$$
\frac{\omega_{2^{-i}}(c_i)}{v_n}=\frac{\omega_{2^{-i}}(c_i-c'_i)+\omega(c'_i)}{v_n}\leq \|c_i-c'_i\|=\|c_i\|-\|c_i'\|\ .
$$
The conclusion follows from this inequality and Equation~\eqref{nonegative}.
\end{proof}

A very similar argument shows that the volume of ``most'' simplices appearing in minimizing sequences converges to $v_n$. We properly state and prove this result, since we will need it later.

\begin{lemma}\label{almost-regular}
Let $c_i$, $i\in\mathbb{N}$ be a minimizing sequence, and for every $i\in\mathbb{N}$ let
$$
c_i=\sum_{k=1}^{n_i} a_{i,k} \sigma_{i,k}\ . 
$$ 
Let $\varepsilon>0$ be fixed, and,
for every $i,k$, set $b_{i,k}=a_{i,k}$ if the hyperbolic volume of a lift of $\sigma_{i,k}$ to $\mathbb{H}^n$ is smaller that $v_n-\vare$,  and $b_{i,k}=0$ otherwise. 
If $c_i'=\sum_{k=1}^{n_i} b_{i,k}\sigma_{i,k}$, then
$$
\lim_{i\to +\infty} \|c'_i\|=0\ .
$$
\end{lemma}
 \begin{proof}

Since the hyperbolic volume of any straight simplex is at most $v_n$, our definition of $c'_i$ implies that 
$$
|\omega_{2^{-i}}(c_i)|\leq |\omega_{2^{-i}} (c_i-c'_i)|+|\omega(c_i')|\leq \|c_i-c'_i\|\cdot v_n+ \|c'_i\|(v_n-\vare)=\|c_i\|v_n-\|c'_i\|\vare\ ,
$$
whence
$$
\|N\|=\lim_{i\to +\infty} \frac{|\omega_{2^{-i}}(c_i)|}{v_n}\leq \lim_{i\to +\infty} \|c_i\| -\frac{\vare}{v_n}\limsup_{i\to +\infty} \|c'_i\|=\|N\|-\frac{\vare}{v_n}\limsup_{i\to +\infty} \|c'_i\|\ .
$$
The conclusion follows.
 \end{proof}

The previous lemma may be exploited to prove that efficient cycles are supported on regular ideal simplices: 
\begin{lemma}[{\cite[Lemma 3.5]{Kue:efficient}}] \label{support}
Let $\mu$ be an efficient cycle for $N$. Then $\mu$ is supported on $\reg(N)\subseteq \overline{S}^*(N)$. 
\end{lemma}

Therefore, we will consider
$\mu$ both as an element of $\calM (\overline{S}^*(N))$ and as an element of $\calM(\reg(N))$.

We are now going to prove that the total variation of an efficient cycle is equal to the simplicial volume of $N$
(recall that the total variation is only lower semicontinuous with respect to weak-* convergence, hence the total variation
of an efficient cycle could be strictly smaller than $\|N\|$ a priori).

To this aim we need the definition of incenter and inradius of a geodesic hyperbolic simplex.
Consider a nondegenerate $n$-simplex $\Delta\in \overline{S}^*({\matH^n})$ (recall that a geodesic simplex is nondegenerate if its image is not contained in a hyperplane).
For every point $p\in \Delta\cap\matH^n$ we denote by $r_\Delta(p)$ the radius of the maximal hyperbolic ball centered in $p$ and contained
in $\Delta$. 
Since the volume of any $n$-simplex is smaller than $v_n$ and
the volume of $3$-balls diverges as the radius diverges, there exists a constant $r_n>0$ such that
$r_\Delta(p)\leqslant r_n$ for every $\Delta\in\overline{S}_n^*({\matH^n})$ and $p\in \Delta$. 

\begin{defn}
Take a nondegenerate simplex $\Delta\in \overline{S}^*_n(\mathbb{H}^n)$. 
The \emph{inradius} $r(\Delta)$ of $\Delta$ is 
$$
r(\Delta)=\sup_{p\in \Delta\cap \matH^n} r_\Delta(p)\ \in \ (0,r_n]
$$ 
(observe that $r(\Delta)>0$ since $\Delta$ is nondegenerate).
The \emph{incenter} $\inc(\Delta)$ is the unique point $p\in \Delta\cap \matH^n$ such that $r_\Delta(p)=r(\Delta)$.
\end{defn}

It is shown in~\cite[Lemma 3.12]{FFM} that the incenter is well-defined, and that the functions
$$
\inc\colon \overline{S}_n^*(\mathbb{H}^n)\to \matH^n, \qquad
r\colon \overline{S}_n^*(\mathbb{H}^n)\to\matR
$$
are continuous. 

If $\Delta$ is a simplex in $\overline{S}^*(N)$, we define its inradius $r(\Delta)$  as the inradius of any lift of $\Delta$ to $\matH^n$, and its
 incenter $\inc(\Delta)$  as the projection in $N$
of the incenter of any lift of $\Delta$ to $\mathbb{H}^n$ (the fact that this notions are well-defined is easily checked).

\begin{lemma}\label{inradius}
Let $\overline{\delta}$ be the inradius of the $n$-dimensional regular ideal simplex, and
let $\Delta_i\in \overline{S}^*_n(\matH^n)$, $i\in\mathbb{N}$ be a sequence such that $\lim_{i\to +\infty} \vol(\Delta_i)=v_n$.
Then $\lim_{i\to +\infty} r(\Delta_i)=\overline{\delta}$.
\end{lemma}
\begin{proof}
By~\cite[Proposition 3.14]{FFM}, for every $i\in\mathbb{N}$ there exists an element $g_i\in\isom(\matH^n)$ such that
$\lim_{i\to +\infty} g_i(\Delta_i)=\overline{\Delta}$, where $\overline{\Delta}$ is a regular ideal simplex. Since the map $r\colon \overline{S}_n^*(\mathbb{H}^n)\to\matR$ is continuous, 
we thus get
$$
\lim_{i\to +\infty} r(\Delta_i)=\lim_{i\to +\infty} r(g(\Delta_i))=r(\overline{\Delta})=\overline{\delta}\ .
$$
\end{proof}

\begin{lemma}\label{compact}
Let $K\subseteq N$ be compact, and let $\delta_0>0$. Then the set
$$
\Lambda=\{\Delta \in \overline{S}^*(N)\, |\, \inc(\Delta)\in K,\, r(\Delta)\geq \delta_0\}
$$
is compact.
\end{lemma}
\begin{proof}
Let $\widetilde{K}\subseteq \mathbb{H}^n$ be a compact subset such that $\pi(\widetilde{K})= K$ (for example, 
if $\pi\colon \mathbb{H}^n\to N$
is the universal covering, then
$\widetilde{K}$ may be chosen as 
the intersection between $\pi^{-1}(K)$ and a Dirichlet  domain for the action of $\Gamma$ on $\mathbb{H}^3$), and let
$$
\widetilde{\Lambda}=\{\widetilde{\Delta} \in \overline{S}^*(\matH^3)\, |\, \inc(\widetilde{\Delta})\in \widetilde{K},\, r(\widetilde{\Delta})\geq \delta_0\}\ .
$$
Under the projection $ \overline{S}^*(\matH^3)\to  \overline{S}^*(N)$, the set $\widetilde{\Lambda}$ is sent to $\Lambda$, hence in order to conclude it suffices to show that
$\widetilde{\Lambda}$ is compact or, equivalently, sequentially compact (being a subset of $(\overline{\mathbb{H}^3})^4$, the space
$\widetilde{\Lambda}$ is metrizable).

Let 
$\widetilde{\Delta}_i=(v_0^i,v_1^i,\dots,v_n^i)\in(\overline{\mathbb{H}}^n)^{n+1}$, $i\in\mathbb{N}$, be a 
sequence of elements in $\widetilde{\Lambda}$.
Since $(\overline{\mathbb{H}}^n)^{n+1}$ is compact, up to passing to a subsequence we may assume that
$\widetilde{\Delta}_i$ tends to $\widetilde{\Delta}_\infty \in (\overline{\mathbb{H}}^n)^{n+1}$. 
Since the maps $r\colon \overline{S}_3^*(\mathbb{H}^3)\to\matR$ and $\inc\colon \overline{S}_3^*(\mathbb{H}^3)\to\matH^3$ are continuous and $\widetilde{K}$ is closed, we have 
$r(\widetilde{\Delta}_\infty)\geq \delta_0$ and $\inc(\widetilde{\Delta}_\infty)\in\widetilde{K}$, thus 
in order to conclude it is sufficient to show that
$\widetilde{\Delta}_\infty$ is nondegenerate. 

To this aim, we may exploit the projective model of hyperbolic space. In this model, each $\widetilde{\Delta}_i$ is a Euclidean simplex. Moreover, since
$\widetilde{K}$ is compact, the positive lower bound on the hyperbolic radii of the hyperbolic balls inscribed in the $\widetilde{\Delta}_i$ provides a positive lower bound on the Euclidean
radius of the Euclidean balls contained in $\widetilde{\Delta}_i$ (hyperbolic balls are represented by Euclidean ellipsoids in the projective model, which in turn contain Euclidean balls). This readily implies that the Euclidean volume of the $\widetilde{\Delta}_i$ is bounded below by a positive constant, hence the limit of the $\widetilde{\Delta}_i$ cannot be degenerate.
\end{proof}

We are now ready to prove Theorem~\ref{totalvariation} from the Introduction, which we recall here for the convenience of the reader:

\begin{varthm}[Theorem~\ref{totalvariation}]
Let $N$ be a complete finite-volume hyperbolic $n$-manifold, $n\geq 3$, and let $\mu$ be an efficient cycle for $N$.
Then $\|\mu\|=\|N\|$.
\end{varthm}
\begin{proof}


Let $r_n$ be a universal upper bound for the inradius of any nondegenerate $n$-dimensional geodesic simplex, as above.
For every $\vare>0$ we set
$$\thi_\varepsilon=\{\Delta \in \overline{S}^*_n(N)\, |\, \inc(\Delta)\in B(N_\varepsilon,r_n)\}\ ,$$
where $B(N_\varepsilon,r_n)$ denotes the closed $r_n$-neighbourhood of $N_\vare$ in $N$; in other words,
$\thi_\vare$ denotes the set of straight nondegenerate simplices of $N$ whose incenter lies in the
closed $r_n$-neighbourhood of the 
 $\vare$-thick part of $N$.
 
Let $\overline{\delta}$ be the inradius of the regular ideal $n$-simplex, and fix any constant $0<\delta_0<\overline{\delta}$. 
Also denote by $V_0$ the hyperbolic volume of a hyperbolic $n$-ball of radius $\delta_0$, and
set $\Omega_{\delta_0}=\{\Delta\in \overline{S}_n^*(N)\, |\, r(\Delta)\geq \delta_0\}$. 

Let now $c_i$, $i\in\mathbb{N}$ be a minimizing sequence, and
let us fix $\vare>0$. We choose $i_0\in\mathbb{N}$ such that $\vol(N\setminus N_{2^{-i_0}})\leq \vare v_n$.
Let $i\geq i_0$, and 
 consider the following partition
of the space of nondegenerate straight simplices in $N$:
$$
\Lambda_1=\overline{S}_n^*(N)\setminus\Omega_{\delta_0} \, ,\quad \Lambda_2=\Omega_{\delta_0}\cap \thi_{2^{-i_0}},\, ,\quad  \Lambda_3=\Omega_{\delta_0}\setminus \thi_{2^{-i_0}}\ .
$$ 

We denote by $c_i=c_i^1+c_i^2+c_i^3$ the corresponding decomposition of $c_i$, i.e.~we assume that the simplices appearing in $c_i^j$ belong to $\Lambda_j$ for $j=1,2,3$. 
By Lemma~\ref{inradius}, since $\delta_0$ is smaller than the inradius
of the regular ideal tetrahedron,
the volume of the lifts to $\matH^n$ of the simplices in $\Lambda_1$ is bounded above by a constant strictly smaller than $v_n$. By Lemma~\ref{almost-regular}, we then have
\begin{equation}\label{facilissima}
\lim_{i\to +\infty} \|c_i^1\|=0\ .
\end{equation}

Let now $\Delta\in \Lambda_3$, i.e. suppose that~$r(\Delta)\geq {\delta_0}$ and 
$\inc(\Delta)\notin B(N_{2^{-i_0}},r_n)$. 
Since $\delta_0\leq r_n$,
the ball $B(\inc(\Delta),\delta_0)\subseteq \Delta$ does not intersect
$N_{2^{-i_0}}$, hence $|\omega_{2^{-i_0}}(\Delta)|\leq v_n-V_0$. Thus
$$
|\omega_{2^{-i_0}} (c_i^3)|\leq \|c_i^3\|\cdot (v_n-V_0)\ .
$$

Observe now that $c_i$, being a relative fundamental cycle for $N_{2^{-i}}$, is \emph{a fortiori} a relative fundamental
cycle for $N_{2^{-i_0}}$. Hence
\begin{align*}
\vol(N_{2^{-i_0}})&=|\omega_{2^{-i_0}}(c_i)|\leq |\omega_{2^{-i_0}}(c_i^1)|+ |\omega_{2^{-i_0}}(c_i^2)|+ |\omega_{2^{-i_0}}(c_i^3)|\\
&\leq \|c_i^1\|\cdot v_n+ \|c_i^2\|\cdot v_n+ \|c_i^3\|\cdot (v_n-V_0)= \|c_i\|\cdot v_n-\|c_i^3\|\cdot V_0\ .
\end{align*}
After dividing by $v_n$ we obtain
$$
\frac{\vol(N_{2^{-i_0}})}{v_n}\leq \|c_i\|-\|c_i^3\|\cdot  \frac{V_0}{v_n}\ ,
$$
whence
$$
\|N\|- \vare=\frac{\vol(N)}{v_n}-\varepsilon \leq \frac{\vol(N_{2^{-i_0}})}{v_n}\leq \|c_i\|-\|c_i^3\|\cdot  \frac{V_0}{v_n}\leq \|N\|+2^{-i}-\|c_i^3\|\cdot  \frac{V_0}{v_n}
$$
and
$$
\|c_i^3\|\leq \frac{v_n}{V_0}(\varepsilon +2^{-i})\ .
$$
In particular, we have
\begin{equation}\label{primastima}
\limsup_{i\to +\infty} \|c_i^3\|\leq \frac{v_n}{V_0}\varepsilon \ .
\end{equation}

Since $\|c_i\|\geq \|N\|$ we also have
$$
\|c_i^2\|= \|c_i\|-\|c_i^1\|-\|c_i^3\|\geq \|N\|-\|c_i^1\|-\frac{v_n}{V_0}(\varepsilon +2^{-i})\ ,
$$
hence (recalling that $\|c_i^1\|\to 0$ as $i\to +\infty$) 
\begin{equation}\label{secondastima}
\liminf_{i\to +\infty} \|c_i^2\|\geq  \|N\| -\vare \frac{v_n}{V_0}\ .
\end{equation}

Observe  that, thanks to Lemma~\ref{compact}, the set $\Lambda_2$ is compact. 
Therefore, one may construct a 
compactly supported function $g\colon \overline{S}^*(N)\to [-1,1]$ such that
$g(\Delta)=1$ for every positively oriented $\Delta \in \Lambda_2$ and 
$g(\Delta)=-1$ for every negatively oriented $\Delta \in \Lambda_2$.
By Lemma~\ref{nonegative:lemma} we then have
$$
\liminf_{i\to +\infty} \int_{\Lambda_2} g \, d\Theta(c_i) =\liminf_{i\to +\infty} \|c_i^2\|\ . 
$$
By definition of weak-* convergence, if $\mu=\lim_{i\to +\infty} \Theta(c_i)$, then from Equations~\eqref{facilissima}, \eqref{primastima} and \eqref{secondastima} (and the fact
that $\|g\|_\infty\leq 1$) we obtain
\begin{align*}
\left| \int_{\overline{S}^*(N)} g\, d\mu \right| &=\left|\lim_{i\to +\infty} \int_{\Lambda_1} g \, d\Theta(c_i)  +\lim_{i\to +\infty} \int_{\Lambda_2} g \, d\Theta(c_i) +
\lim_{i\to +\infty} \int_{\Lambda_3} g \, d\Theta(c_i)\right|\\
&\geq -\left| \lim_{i\to +\infty} \int_{\Lambda_1} g \, d\Theta(c_i)\right|  + \left| \lim_{i\to +\infty} \int_{\Lambda_2} g \, d\Theta(c_i) \right|-
\left|\lim_{i\to +\infty} \int_{\Lambda_3} g \, d\Theta(c_i)\right|\\
&\geq -\limsup_{i\to +\infty} \|c_i^1\|+\liminf_{i\to +\infty} \|c_i^2\|-\limsup_{i\to +\infty} \|c_i^3\|\\
& \geq \|N\|- \frac{2\vare v_n}{V_0}\ .
\end{align*}
Since $\|g\|_\infty\leq 1$, this inequality  implies that the total variation of $\mu$ is not smaller than $\|N\|- \frac{2\vare v_n}{V_0}$.
Due to the arbitrariness of $\vare$, we may conclude that $\|\mu\|\geq \|N\|$.
On the other
hand, it is well known that the total variation is lower semicontinuous with respect to the weak-* convergence,
hence $\|\mu\|\leq \lim_{i\to +\infty} \|\Theta(c_i)\|=\|N\|$. This concludes the proof.
\end{proof}

Our normalization of the Haar measure implies that $\|\mu_\eq\|=2\vol(N)$. Therefore, Theorem~\ref{totalvariation} readily implies the following:

\begin{cor}\label{cor}
Let $k\in\mathbb{R}$ and suppose that the measure $\mu=k\mu_\eq$ is an efficient cycle for $N$. Then $k=1/(2v_n)$.
\end{cor}

\section{Manifolds admitting a unique efficient cycle}\label{unique:sec}
Jungreis and Kuessner proved that, if $N$ is a non-Gieseking like hyperbolic manifold, then every efficient cycle of $N$ is equidistributed. 
In this section we strengthen this result by showing that the same conclusion holds under the {supposedly} less restrictive requirement  that $N$ be non-commensurable
with the Gieseking manifold. We may concentrate our attention on the three-dimensional case, the higher-dimensional case being covered by the results
proved in~\cite{Kue:efficient}.
Therefore, throughout this section we denote by  $N$ a complete finite-volume hyperbolic $3$-manifold. 

Henceforth we fix a regular ideal simplex (with ordered vertices) $\Delta_0\in\reg(\matH^3)$, which we exploit to fix an identification 
$\reg(\matH^3)\cong \isom(\matH^3)$, as explained at the end of Section~\ref{preliminaries}.
For $i=0,\dots,3$,
let $r_i\in \isom^-(\matH^3)$ be the hyperbolic reflection with respect to the plane containing the $i$-th face of $\Delta_0$. Under the identification
$\reg( \matH^3)\cong \isom(\matH^3)$, the \emph{right} multiplication 
$$
\isom(\matH^3)\to \isom(\matH^3)\, ,\qquad g\mapsto g\cdot r_i
$$
corresponds to the map $\rho_i\colon \reg( \matH^3)\to \reg( \matH^3)$ sending any simplex $\Delta\in\reg(\mathbb{H}^3)$ to the simplex obtained by reflecting $\Delta$ with respect to plane containing
its $i$-th face. 
We denote by $R$ the subgroup of $\isom(\matH^3)$ generated by the $r_i$, $i=0,\dots,3$, and we set $R^\pm=R\cap \isom^\pm(\matH^3)$.
Observe that, since the (left) action of $\Gamma$ and the (right) action of $R$ on $\isom(\matH^3)$ commute, the groups $R,R^+$ also act on
$\reg(N)$.

Recall that any efficient cycle for $N$ is supported on $\reg(N)$, so that  we can consider
 efficient cycles as elements of $\calM(\reg(N))$. 
The following result is proved in~\cite[Lemmas 3.9 and 3.10]{Kue:efficient} (see also~\cite[Lemma 2.2]{Jungreis}):

\begin{lemma}\label{reflection:lemma}
Let $\mu\in \calM(\reg(N))$ be any efficient cycle for $N$. Then $\mu$ is invariant with respect to the right action of $R^+$
on 
$\reg(N)$. Moreover, for every $r\in R^-$ we have $r_*(\mu)=-\mu$. 
\end{lemma}

If $\Delta=(v_0,v_1,v_2,v_3)\in\reg(\matH^3)$ is an arbitrary regular ideal simplex, we denote by $T_\Delta\subseteq \reg(\matH^3)$  the set defined as follows:
$\Delta'=(v'_0,v'_1,v'_2,v'_3)\in \reg(\matH^3)$ belongs to $T_\Delta$ if and only if its vertices $v'_0,v'_1,v'_2,v'_3$ span a geodesic simplex of the unique tiling of $\matH^3$
by regular ideal tetrahedra
containing the geodesic simplex spanned by $v_0,v_1,v_2,v_3$. We also denote by $\aut(T_\Delta)<\isom(\matH^3)$ the subgroup of $\isom(\matH^3)$ leaving
$T_\Delta$ invariant. It is easy to check that $\aut(T_\Delta)$ is discrete. 

Recall that two subgroups $\Gamma_1,\Gamma_2$ of $\isom(\matH^3)$ are \emph{commensurable} if there exists 
$g\in\isom(\matH^3)$ such that $(g\Gamma_1g^{-1})\cap \Gamma_2$ has finite index both in $g\Gamma_1g^{-1}$ and in $\Gamma_2$. If $\Gamma_1$ and $\Gamma_2$ are discrete and torsion-free,
this is equivalent to requiring that the hyperbolic manifolds $\Gamma_1\backslash \mathbb{H}^3$ and $\Gamma_2\backslash \mathbb{H}^3$
admit a common finite-sheeted Riemannian covering. 

\begin{lemma}\label{commens1}
For every $\Delta\in\reg(\matH^3)$, the group $\aut(T_\Delta)$ is commensurable with $R$. Moreover, both these groups are commensurable with the
fundamental group of the Gieseking manifold.
\end{lemma}
\begin{proof}
 If $g\in\isom(\matH^3)$ is such that $g(\Delta_0)=\Delta$, then $g\cdot   \aut(T_{\Delta_0}) \cdot g^{-1}=\aut(T_{\Delta})$. Moreover,
$R<\aut(T_{\Delta_0})$ and the index of $R$ is finite, since $\aut(T_{\Delta_0})$ is discrete and $R$ has finite covolume.
This implies that $\aut(T_\Delta)$ is commensurable with $R$. 

Up to conjugacy, we may suppose that $\Delta_0$ is a fundamental domain for the action of the fundamental group $G$ of the Gieseking manifold on $\matH^3$. 
Then $G<\aut(T_{\Delta_0})$ and, as above, the index of $G$ in $\aut(T_{\Delta_0})$ is finite because $G$ has finite covolume (in fact, this index
is equal to $4!=24$). This concludes the proof.

\end{proof}


\begin{thm}\label{ergodic:prop}
Suppose there exists a non-equidistributed efficient cycle  $\mu\in \calM(\reg(N))$. Then $N$ is commensurable with the Gieseking manifold. 
\end{thm}
\begin{proof}
Let $N=\Gamma\backslash \matH^3$.
As proved in~\cite[Section 4]{Kue:efficient}, the efficient cycle $\mu$ decomposes into the sum of a multiple
of $\mu_\eq$ and a measure $\mu'\in\reg(N)$ which is supported on tetrahedra whose lifts in $\matH^3$
have all their vertices in parabolic fixed points of $\Gamma$. 
Since $\mu$ is non-equidistributed, we may assume
that $\mu'\neq 0$.

Since parabolic fixed points of $\Gamma$ are in countable number, the support of $\mu'$ is also countable, and this implies in turn that $\mu'$ is purely atomic.
Moreover, since $\mu'=\mu-k\mu_\eq$ for some $k\in\mathbb{R}$, the measure $\mu'$ also satisfies $r_*(\mu')=\mu'$ for every $r\in R^+$
and $r_*(\mu')=-\mu'$ for every $r\in R^-$. Let us set 
$$
\Omega=\{[\Delta]\in \reg(N)\, |\, \mu'(\{ [\Delta]\})\neq 0\}\neq \emptyset\ .
$$
Due to the $R$-equivariance of $\mu'$, the countable set $\Omega$ is $R$-invariant. Let us fix a  non-empty $R$-orbit $\overline{\Omega}\subseteq \Omega$.
The absolute value of the measure $\mu'$ is constant on elements of $\overline{\Omega}$. Since $\mu'$ has finite total variation, this implies that 
$$\overline{\Omega}=\{[\Delta_1],\dots,[\Delta_k]\}$$ 
is finite. For every $i=1,\dots, k$, let $\Delta_i$ be a lift of $[\Delta_n]\in \reg(N)$ in $\reg(\matH^3)$. 
By looking at the definition of the actions of $R$ and of $\Gamma$ on $\reg(\matH^3)$, we deduce that
the $R$-orbit of $\Delta_1$ in $\reg(\matH^3)$ is contained in
$$
\Gamma\cdot \Delta_1\cup\dots\cup \Gamma\cdot \Delta_k\ .
$$
Observe now that the $R$-orbit of $\Delta_1$ in $\reg(\matH^3)$  realizes a tiling of $\matH^3$ by regular ideal tetrahedra. Therefore, up to adding to the $\Delta_j$, $j=1,\dots,k$ all the simplices obtained by permuting their vertices (which are still in finite number), we may
assume that 
\begin{equation}\label{union}
T_{\Delta_1} \subseteq 
\Gamma\cdot \Delta_1\cup\dots\cup \Gamma\cdot \Delta_k \ .
\end{equation}

We are now going to show that the group $\Gamma\cap \aut(T_{\Delta_1})$ has finite index in $\aut(T_{\Delta_1})$. 
To this aim we will just exploit~\eqref{union}; hence,
for every  $j=1,\dots,k$ we may assume that $T_{\Delta_1}\cap (\Gamma\cdot \Delta_j)\neq \emptyset$. Thus, up to replacing $\Delta_j$ 
with another simplex in its $\Gamma$-orbit, we suppose that $\Delta_j\in T_{\Delta_1}$. Observe now that $T_{\Delta_1}$ is the orbit of $\Delta_1$
under the action of $\aut(T_{\Delta_1})$, hence, thanks to~\eqref{union}, for every $j=1,\dots,k$ there exists $g_j\in \aut(T_{\Delta_1})$
such that $g_j\cdot \Delta_j=\Delta_1$. 

Let us fix $g\in \aut(T_{\Delta_1})$. Since $g\cdot \Delta_1 \in T_{\Delta_1}\subseteq \Gamma\cdot \Delta_1\cup\dots\cup \Gamma\cdot \Delta_k$, there exist
$\gamma\in \Gamma$, $j\in\{1,\dots,k\}$ such that $g\cdot \Delta_1=\gamma\cdot \Delta_j$, whence $(\gamma^{-1}g)\cdot \Delta_1=\Delta_j$
and $(g_j\gamma^{-1}g)\cdot \Delta_1=g_j\cdot \Delta_j=\Delta_1$. However, since the unique hyperbolic isometry which fixes
the vertices of a regular ideal tetrahedron is the identity, the stabilizer of $\Delta_1$ in $\aut(T_{\Delta_1})$ is trivial, hence
$g_j\gamma^{-1}g=1$, i.e.~$g=\gamma g_j^{-1}$ (and, in particular, $\gamma\in \Gamma\cap \aut(T_{\Delta_1})$). We have thus shown that the set $\{g_1,\dots,g_k\}$
contains a set of representatives for the set of right lateral classes of $ \Gamma\cap \aut(T_{\Delta_1})$ in $\aut(T_{\Delta_1})$.

Since $\Gamma$ is discrete and $ \Gamma\cap \aut(T_{\Delta_1})$ has finite covolume (being a finite index subgroup of  $\aut(T_{\Delta_1})$),
the group  $ \Gamma\cap \aut(T_{\Delta_1})$ has finite index also in $\Gamma$. Thus $\Gamma$ is commensurable with  $ \aut(T_{\Delta_1})$,
hence $N$ is commensurable with the Gieseking manifold by Lemma~\ref{commens1}.

\end{proof}

Putting together Theorem~\ref{ergodic:prop} and Corollary~\ref{cor} we obtain the following:

\begin{cor}\label{description}
Let $N$ be a complete finite-volume $n$-hyperbolic manifold, $n\geq 3$, and suppose that $N$ is not commensurable with the Gieseking manifold. 
Then $N$ admits a unique efficient cycle, which is given by the measure
$$
\frac{1}{2 v_n} \cdot \mu_\eq\ .
$$
\end{cor}

\section{Manifolds admitting non-equidistributed efficient cycles}\label{nonunique:sec}

In this section we prove that manifolds that are commensurable with the Gieseking manifold admit non-equidistributed efficient cycles. 
We will first prove that this phenomenon occurs for manifolds admitting an ideal triangulation by regular ideal tetrahedra, and we will then deduce the general
case from the fact that any manifold which is commensurable with the Gieseking manifold admits a finite covering with such a triangulation.

\subsection{Triangulations and ideal triangulations}
Let $\overline{N}$ be a compact 3-manifold with non-empty boundary made of tori. We recall the well-known notions of triangulation and ideal triangulation, widely used in 3-dimensional topology.

A \emph{triangulation} of $\overline{N}$ is a realization of $\overline{N}$ via a simplicial face-pairing of finitely many tetrahedra. A triangulation of $\overline{N}$ naturally induces a triangulation of its boundary. 

An \emph{ideal triangulation} of $\overline{N}$ (or of $N$) is a realization of $N=\inte{\overline{N}}$ as a simplicial face-pairing of finitely many tetrahedra, with all their vertices removed. The removed vertices are called \emph{ideal} and they correspond to the boundary components of $\overline{N}$; the link of every ideal vertex is a triangulation of the corresponding boundary component of $\overline{N}$.

We say as usual that $\overline{N}$ is \emph{hyperbolic} if its interior has a finite-volume complete hyperbolic metric. If $\overline{N}$ is hyperbolic, every geometric decomposition of $N$ into hyperbolic ideal tetrahedra is an example of ideal triangulation, that we call a \emph{geometric ideal triangulation} of $\overline{N}$ (or of $N$). We still do not know if every hyperbolic 3-manifold has a geometric ideal triangulation, but we know it does so virtually \cite{LST}.

We are interested here in transforming a geometric ideal triangulation into a triangulation in an efficient way. One method called \emph{inflation} was introduced by Jaco and Rubinstein in \cite{JR}. Here we introduce a similar method where we employ the dual viewpoint of simple spines, as  Matveev in \cite{Mat, Mat:book}, in a similar fashion as in~\cite{FFM}.

Consider a geometric ideal triangulation $T$ of $N$. We lift it to a geometric ideal triangulation $\widetilde T$ of the universal cover $\matH^3$. We choose some disjoint cusp sections in $N$; their pre-image consists of infinitely many disjoint horoballs in $\matH^3$, centered at the vertices of $\widetilde T$. 

For $\vare>0$ sufficiently small, the $\vare$-thick part $N_\vare$ of $N$ is obtained by removing from $N$ sufficiently deep cusp sections, and it is homeomorphic to $\overline{N}$. 
The ideal triangulation $T$ of $N$ restricts to a decomposition of $N_\vare$ into truncated tetrahedra. To obtain a triangulation for $N_\vare$  would now suffice to take its barycentric subdivision; however, this operation is not useful for us because it produces too many tetrahedra: we are looking for a triangulation for $N_\vare$ which contains the same number of tetrahedra as $T$, plus only a few more.

We explain our request more precisely. We say that a triangulation $T'$ of $N_\vare$ is \emph{adapted} to the geometric ideal triangulation $T$ if there is an injective map $i$ from the set of ideal tetrahedra of $T$ to the set of tetrahedra of $T'$ such that for every tetrahedron $\Delta$ of $T$, every lift of $i(\Delta)$ is a tetrahedron in $\matH^3$ whose vertices lie in the boundary of 4 removed horoballs whose centers are the vertices of a lift of $\Delta$. (We do not require the lift of $i(\Delta)$ to be a straight tetrahedron, only a topological one.) In some sense we require $\Delta$ and $i(\Delta)$ to be close. Every tetrahedron of $T'$ that is not in the image of $i$ is called \emph{residual}. 

We will need the following lemma, which says that for any hyperbolic manifold $N$ with a geometric ideal triangulation $T$ it is possible to construct a tower of finite coverings, each equipped with an adapted triangulation $T_i'$ whose residual tetrahedra grow sublinearly with respect to the degree of the cover.

\begin{prop} \label{tower:prop}
Let $N$ be a hyperbolic manifold equipped with a geometric ideal triangulation $T$. There is a tower of finite coverings $W_i \to \overline{N}$ of degree $d_i$ such that the following holds: every $W_i$ admits a triangulation $T_i'$ adapted to the geometric ideal triangulation $T_i$ obtained by lifting $T$, with $r_i$ residual tetrahedra, such that
$$\lim_{i\to \infty} \frac {r_i}{d_i} \to 0.$$
\end{prop}

Subsections~\ref{adapted:subsection} and~\ref{char:sub} are devoted to a proof of this proposition.
\subsection{Construction of an adapted triangulation} \label{adapted:subsection}
We introduce an efficient method to transform a geometric ideal triangulation $T$ of a hyperbolic $N$ into a triangulation $T'$ that is adapted to $T$.

\begin{figure}
 \begin{center}
  \includegraphics[width = 9 cm]{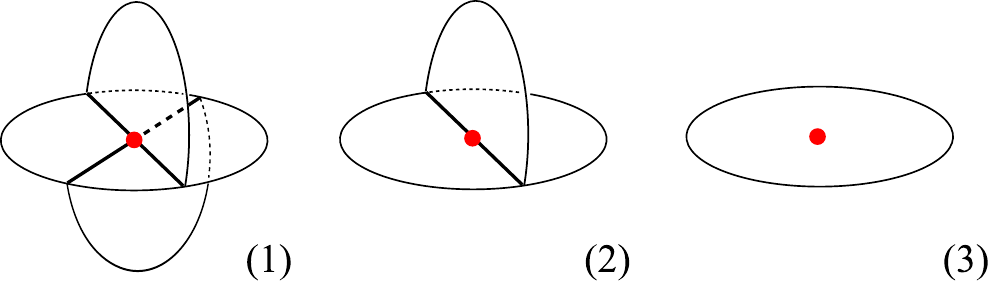}
 \end{center}
 \nota{Neighborhoods of points in a simple polyhedron.}
 \label{models:fig}
\end{figure}

A compact $2$-dimensional polyhedron $X$ is \emph{simple} if every point $x$ of $X$ has a star neighborhood PL-homeomorphic to one of the three models shown in Figure \ref{models:fig}. Points of type (1) are called \emph{vertices}. Points of type (2) and (3) form respectively some manifolds of dimension 1 and 2: their connected components are called respectively \emph{edges} and \emph{regions}. A simple polyhedron $X$ is \emph{special} if every edge is an open segment and every region is an open disc, so in particular it has a natural CW structure.

Let $\overline{N}$ be a compact $3$-manifold with (possibly empty) boundary. A compact 2-dimensional subpolyhedron $X\subset N= \inte{\overline{N}}$ is a \emph{spine} of $\overline{N}$ if $\overline{N}\setminus X$ consists of an open collar of $\partial \overline{N}$ and some (possibly none) open balls (the presence of some open balls is necessary when $\partial \overline{N} = \emptyset$). 

\begin{figure}
\begin{center}
\includegraphics[width = 8 cm] {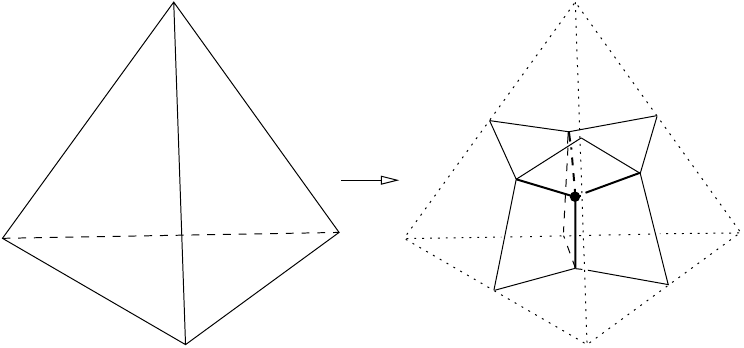}
\nota{By dualizing an ideal triangulation we get a simple spine.}
\label{dualspine:fig}
\end{center}
\end{figure}

Let $\overline{N}$ be a compact manifold with boundary made of tori. Suppose that $N$ is hyperbolic and equipped with a geometric ideal triangulation $T$. We now describe a method to construct a triangulation $T'$ for $\overline{N}\cong N_\vare$ adapted to $T$.

First, we dualise the ideal triangulation $T$ to get a special spine $X$ of $\overline{N}$ with one vertex at the barycenter of each ideal tetrahedron as shown in Figure \ref{dualspine:fig}.

Second, we add some cells to $X$ to obtain a new special polyhedron $X'$, so that by dualizing $X'$ back we will get our desired triangulation $T'$. We construct $X'$ as follows. By construction $\overline{N}\setminus X$ consists of an open collar of $\partial \overline{N}$, that is a finite union of products $S\times (0,1]$ where $S$ is a torus and $S \times \{1\}$ is a boundary component of $N$. 
Choose a $\theta$-shaped graph $Y \subset S$ that is itself a spine of $S$, \emph{i.e.}~$S\setminus Y$ consists of an open disc. Add to $X$ the polyhedron
$$Y \times (0,1] \bigcup S \times \{1\}.$$

If we do this at each product $S\times (0,1]$ in $\overline{N}\setminus X$, we obtain a 2-dimensional polyhedron $X' \subset \overline{N}$ that contains $\partial \overline{N}$. If $Y$ is chosen generically, the polyhedron $X'$ is special. The complement $\overline{N}\setminus X'$ consists of open balls, one for each boundary component of $\overline{N}$.

As we mentioned above, the triangulation $T'$ for $\overline{N}$ is constructed by dualizing $X'$ in the appropriate way. Every boundary torus $S$ of $\overline{N}$ inherits from $X'$ a cellularization with two vertices, three edges, and one disc (the cellularization depends on the chosen $\theta$-shaped spine $Y$); this cellularization is dualised to a one-vertex triangulation for $S$. This triangulation extends from $\partial \overline{N}$ to $\overline{N}$ as follows: every disc, edge, and vertex of $X'$ that is not adjacent to $\partial \overline{N}$ dualizes to an edge, a triangle, and a tetrahedron for $T'$.

The resulting triangulation $T'$ has the smallest possible number of vertices: one for each boundary component. The tetrahedra of $T$ are in natural 1-1 correspondence with the vertices of $X$. The tetrahedra of $T'$ are in natural 1-1 correspondence with the vertices of $X'$ that are not contained in $\partial \overline{N}$. Since every vertex of $X$ is also a vertex of $X'$ of this kind, we get a natural injection $i$ from the set of tetrahedra of $T$ into the set of tetrahedra of $T'$.

\begin{lemma}
If $T$ is a geometric ideal triangulation for $\overline{N}$, the triangulation $T'$ is adapted to $T$.
\end{lemma}
\begin{proof}
We fix some disjoint horocusp sections and truncate $N$ along these, to obtain a smaller copy $N_\vare$ of $\overline{N}$.
Their preimage in $\matH^3$ consists of horospheres. When passing from $X$ to $X'$ we add the cusp sections $\partial N_\vare$ and some products $Y\times (0,1]$. In $\matH^3$ this corresponds to adding the horospheres and some products $\widetilde Y \times (0,1]$. The resulting dual triangulation $T'$ has all its vertices in the cusp sections, which lift to vertices in the horospheres. By construction for every ideal tetrahedron $\Delta$ in $T$ the corresponding $i(\Delta)$ has its vertices in the same horospheres that are crossed by the edges of $\Delta$.
\end{proof}

The residual tetrahedra correspond to the vertices of $X'$ contained in the interior of $\overline{N}$ that were not themselves vertices of $X$, and that were created by attaching the products $Y \times (0,1]$ along some generic map $Y \to X$. We now need to construct some tower of coverings where this kind of vertices grow sublinearly in number.

\subsection{Characteristic coverings}\label{char:sub}
We now build the tower of coverings for $\overline{N}$ needed in Proposition \ref{tower:prop}. 
We will use some results of Hamilton \cite{Ham} on coverings determined by characteristic subgroups. A similar construction was made in \cite[Section 5.3]{FFM}.

Recall that a \emph{characteristic subgroup} of a group $G$ is a subgroup $H<G$ which is invariant by any automorphism of $G$. For a natural number $x\in \matN$, the \emph{$x$-characteristic} subgroup of $\matZ\times \matZ$ is the subgroup $x(\matZ\times\matZ)$ generated by $(x,0)$ and $(0,x)$. It has index $x^2$ if $x>0$ and $\infty$ if $x=0$. The characteristic subgroups of $\matZ\times \matZ$ are precisely the $x$-characteristic subgroups with $x\in \matN$. It is easy to prove that a subgroup of $\matZ\times\matZ$ of index $x$ contains the $x$-characteristic subgroup.

A covering $p\colon \widetilde T\to T$ of tori is \emph{$x$-characteristic} if $p_*(\pi_1(\widetilde T))$ is the $x$-characteristic subgroup of $\pi_1(T)\cong \matZ\times \matZ$. A covering $p\colon \widetilde {\overline{N}} \to \overline{N}$ of 3-manifolds bounded by tori is \emph{$x$-characteristic} if the restriction of $p$ to each boundary component of $\widetilde{\overline{N}}$ is $x$-characteristic.

Lemma 5 from \cite{Ham} implies the following.

\begin{lemma}[E. Hamilton] \label{Hamilton:lemma}
Let $\overline{N}$ be a hyperbolic compact, orientable 3-manifold with boundary tori. For every integer $i>0$ there exist an integer $k>0$ and a finite-index normal subgroup $K_i\triangleleft \pi_1(\overline{N})$ such that $K_i\cap \pi_1(T_{j})$ is the characteristic subgroup of index $(ik)^2$ in $\pi_1(T_{j})$, for each component $T_{j}$ of $\partial \overline{N}$. Hence the covering $W_i \to \overline{N}$ corresponding to $K_i$ is $(ik)$-characteristic.
\end{lemma}

We can now prove Proposition \ref{tower:prop}. We restate it for the sake of clarity.

\begin{varthm}[Proposition~\ref{tower:prop}]
Let $N$ be a hyperbolic manifold equipped with a geometric ideal triangulation $T$. There is a tower of finite coverings $W_i \to \overline{N}$ of degree $d_i$ such that the following holds: every $W_i$ admits a triangulation $T_i'$ adapted to the geometric ideal triangulation $T_i$ obtained by lifting $T$, with $r_i$ residual tetrahedra, such that
$$\lim_{i\to \infty} \frac {r_i}{d_i} \to 0.$$
\end{varthm}

\begin{proof}
Let $X$ be the spine dual to $T$. Following Section \ref{adapted:subsection} we enlarge $X$ to a special polyhedron $X'$ by adding one piece
$$Y \times (0,1] \bigcup S \times \{1\}$$
for each boundary torus $S$ of $\overline{N}$, inside the corresponding collar $S \times (0,1]$ in $\overline{N}\setminus X$. This operation depends on the choice of a generic $\theta$-shaped spine $Y\subset S$.

\begin{figure}
\begin{center}
\includegraphics[width = 12.5 cm] {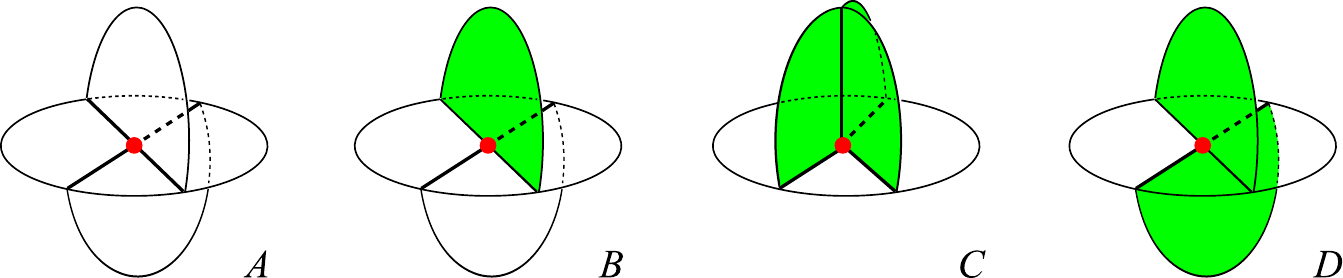}
\nota{We colour in green the regions of the inserted portions $Y\times (0,1) \cup S \times \{1\}$. There are four types of vertices $A$, $B$, $C$, and $D$ in the spine $Q$, according to the colours of the incident regions.}
\label{types:fig}
\end{center}
\end{figure}

The polyhedron $X'$ has all the vertices of $X$, plus some additional ones that we now investigate carefully. The following discussion is similar to \cite[Section 5.4]{FFM}. Colour in white the regions of $X$ and in green the regions in the products $Y\times (0,1] \cup S \times \{1\}$ that are attached to $X$. There are five types $A, B, C, D, E$ of vertices in $X'$ according to the colours of the incident regions: the vertices of type $A,B,C,D$ are shown in Figure~\ref{types:fig}, while those of type $E$ are those that lie in $\partial \overline{N}$ and that are incident to green regions only. The vertices of type $A$ are precisely those of $X$. The vertices of type $B,C,D$ are dual to the residual tetrahedra of $T'$, and we want to control their number. Those of type $E$ are not interesting here.

For every boundary torus $S$, the collar map $S \to X$ is a (possibly non-injective) immersion, and the cellularization of $X$ pulls back to a cellularization of $S$, which is in fact dual to the triangulation link of the corresponding ideal vertex of $T$. The $\theta$-shaped spine $Y$ is generic, transverse to this cellularization as in Figure \ref{tori:fig} (left). The four types of vertices $A,B,C,D$ that may arise are shown in Figure \ref{tori:fig} (right).

\begin{figure}
\begin{center}
\includegraphics[width = 9 cm] {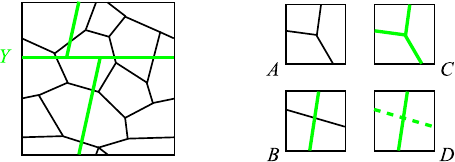}
\nota{The cellularization of a boundary torus $S$ induced by the collar map $S\to X$, and the $\theta$-shaped spine $Y$ of $S$ coloured in green (left). The four types of vertices $A$, $B$, $C$, $D$ (right).}
\label{tori:fig}
\end{center}
\end{figure}

Let $v_A$, $v_B$, $v_C$, and $v_D$ be the number of vertices of type $A$, $B$, $C$, and $D$ in $X'$. The number of residual tetrahedra in $T'$ is $v_B+v_C+v_D$.

We build the tower of coverings. By Lemma \ref{Hamilton:lemma}, for every integer $i\geq 1$, there are a $k_i>0$ and an $(ik_i)$-characteristic covering $W_i \to \overline{N}$. 

We now construct the triangulation $T_i'$ adapted to the lifted geometric ideal triangulation $T_i$ of $W_i$. The pre-image of $X$ is a spine $X_i$ of $W_i$ dual to $T_i$.
To construct the adapted triangulation $T_i'$, we choose an appropriate $\theta$-shaped spine inside every boundary torus of $W_i$. We explain now how to make this choice.

\begin{figure}
\begin{center}
\includegraphics[width = 12.5 cm] {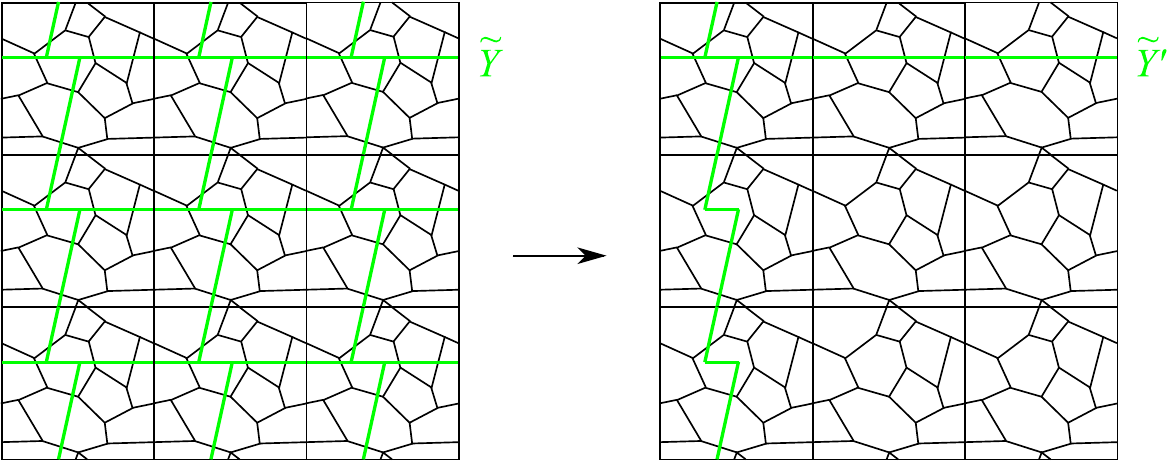}
\nota{A 3-characteristic covering $\widetilde S \to S$. The spine $Y$ of $S$ lifts to the green spine $\widetilde Y$ shown in the left picture. We can eliminate most of its edges and still get a spine $\widetilde Y'$ of $\widetilde S$.}
\label{tori2:fig}
\end{center}
\end{figure}

Since the covering $W_i \to \overline{N}$ is $(ik_i)$-characteristic, every boundary torus $\widetilde S$ of $W_i$ covers a torus $S$ of $\overline{N}$ as an $(ik_i)$-characteristic covering. The case $ik_i=3$ is shown in Figure \ref{tori2:fig}.
We have chosen in the previous paragraphs a spine $Y$ for $S$, see Figure \ref{tori:fig}. As shown in Figure~\ref{tori2:fig}-(left), the pre-image $\widetilde Y$ of $Y$ in $\widetilde S$ is a spine of $\widetilde S$, whose complement in $S$ consists of $(ik_i)^2$ discs. Figure \ref{tori2:fig}-(right) shows that we can eliminate most vertices and edges of $\widetilde Y$ and obtain a simpler spine $\widetilde Y'\subset \widetilde Y$ of  $\widetilde S$, whose complement in $\widetilde S$ consists of only one disc. This is the $\theta$-shaped spine that we use on each boundary component $\widetilde S$ of $W_i$. 

It remains to estimate the number $r_i$ of residual tetrahedra in $T_i'$. Recall that
$$r_i = v_B^i + v_C^i + v_D^i$$
where $v_B^i, v_C^i, v_D'$ are the numbers of vertices of type $B, C, D$ in the dual polyhedron $X_i'$.  The covering $W_i \to \overline{N}$ has degree 
$$d_i = (ik_i)^2h_i$$ 
where $h_i$ is the number of distinct boundary tori in $\partial W_i$ that project to one boundary torus of $\overline{N}$. 
It is clear from Figure \ref{tori2:fig} that
$$v_B^i \leq 2ik_ih_iv_B, \quad v_C^i \leq 2v_C, \quad v_D^i \leq 2ik_ih_iv_D.$$
Therefore
$$\frac{r_i}{d_i} = \frac{v_B^i + v_C^i + v_D^i}{(ik_i)^2h_i} \leq \frac{2ik_ih_i}{(ik_i)^2h_i} (v_B+v_C+v_D) \longrightarrow 0$$
as $i\to \infty$. The proof is complete.
\end{proof}

\subsection{Efficient cycles from regular ideal triangulations}\label{mu:subsec}
We are now ready to show that if a hyperbolic $3$-manifold $N$ admits a geometric ideal triangulation $T$ by regular ideal tetrahedra, then it also admits 
a non-equidistributed efficient cycle. Indeed, let $\Delta_1,\Delta_2,\dots, \Delta_h$ be the regular ideal tetrahedra of $T$, considered as subsets of $N$. 
For every $i=1,\dots,k$ we denote by $\widetilde{\sigma_i}\in\reg^+(\matH^3)\subseteq (\mathbb{H}^3)^4$ a (positively oriented) ordering $(\widetilde{v}_0,\dots,\widetilde{v}_3)$ of
the set of vertices of a lift of $\Delta_i$ to $\matH^3$, and by $\sigma_i$ the class of $\widetilde{\sigma}_i$ in $\reg^+(N)$.   
Finally, we set
$$
\mu_T=\Theta \left(\alt_3\left(\sum_{i=1}^k \sigma_i\right)\right)\ .
$$
(Strictly speaking, we defined the alternating operator only on simplices with vertices in $\matH^3$, but of course it may be extended by the same formula also on ideal simplices). 

The main result of this section is Theorem~\ref{explicit:thm}, which we recall here for the convenience of the reader:

\begin{varthm}[Theorem~\ref{explicit:thm}]
Let $N$ be a complete finite-volume $3$-manifold admitting a decomposition $T$ into regular ideal tetrahedra. Then $\mu_T$ is an efficient cycle for $N$.
\end{varthm}
\begin{proof}
Let us fix some notation. As usual, for every sufficiently large $i\in\mathbb{N}$ we fix an identification $\overline{N}\cong N_{2^{-i}}$ between the natural compactification of $N$ and 
the $2^{-i}$-thick part of $N$. 
By Proposition~\ref{tower:prop}, there is a tower of finite coverings $W_i \to N_{2^{-i}}$ of degree $d_i$ such that the following holds: every $W_i$ admits a triangulation $T_i'$ adapted to the geometric ideal triangulation $T_i$ obtained by lifting $T$, with $r_i$ residual tetrahedra, such that
\begin{equation}\label{pochi}
\lim_{i\to \infty} \frac {r_i}{d_i} \to 0\ .
\end{equation}

For every sufficiently large $i\in\mathbb{N}$, we construct a relative fundamental cycle $c_i$ for $N_{2^{-i}}$ as follows. The universal covering of $W_i$ coincides with the universal covering
of $N_{2^{-i}}$ (which is the complement of a collection of disjoint horoballs in $\matH^3$), hence we may apply the straightening operator to any positively-oriented parametrization
of any simplex appearing in $T'_i$; after applying the alternating operator to the sum of the obtained straight tetrahedra, we get  a relative fundamental cycle
$\widetilde{c}_i$ for $W_i$ (more precisely, for the pair $(W'_i,W'_i\setminus \inte{W_i})$, where $W'_i$ is the complete finite-volume hyperbolic manifold obtained
from $W_i$ by adding back the removed cusps). If $p_i\colon (W'_i,W'_i\setminus \inte{W_i})\to (N,N\setminus \inte{N_{2^{-i}}})$ is the covering projection, we then set
$$
c_i=\frac{(p_i)_* (\widetilde{c}_i)}{d_i}\ .
$$
For simplicity, we will say that a simplex appearing in $c_i$ is \emph{non-residual} if it is obtained (via $(p_i)_*$) from the alternation of the straightening
of a non-residual simplex of $T_i'$.

It is easy to check that $c_i, i\in\mathbb{N}$ is a minimizing sequence: indeed, if $k$ is the number of the tetrahedra of $T$, then
$\vol(N)=kv_3$, hence $\|N\|=\vol(N)/v_3=k$. On the other hand, by construction the number of non-residual simplices in $\widetilde{c}_i$ is equal to $kd_i$ and the alternating operator
is norm non-increasing, hence 
$$
\limsup_{i\to +\infty} \|c_i\|=\limsup_{i\to +\infty} \frac{\|(p_i)_* (\widetilde{c}_i)\|}{d_i}\leq \limsup_{i\to +\infty} \frac{ \|\widetilde{c}_i\|}{d_i}=\limsup_{i\to +\infty} \frac{ k d_i +r_i }{d_i}= k\ ,
$$
and this proves that the sequence $c_i$, $i\in\mathbb{N}$ is minimizing.

In order to conclude we are then left to show that
$$
\lim_{i\to +\infty} \Theta (c_i)=\mu_T\ .
$$
Let $\Delta_0\in \reg^+(N)$ be a (positively oriented representative of a) tetrahedron of $T$, and let $\widetilde{\Delta}_0\in \reg(\matH^3)$ be a lift of $\Delta_0$ to $\matH^3$ with vertices
$(v_0,v_1,v_2,v_3)$. There exist pairwise disjoint open neighbourhoods $U_0,\dots,U_3$ of $v_0,\dots,v_3$ in $\overline{\matH}^3$ such that the following conditions hold:
every straight tetrahedron having its $i$-th vertex in $U_i$ is nondegenerate and positively oriented, and the tetrahedron
$\widetilde{\Delta}_0=(v_0,v_1,v_2,v_3)$ is the unique lift of tetrahedra of $T$ whose vertices lie (in the correct order) in $U_0,\dots,U_3$. We set
$$
\widetilde{\Omega}=\{(v_0,v_1,v_2,v_3)\in \overline{S}_3^*({\matH}^3)\, |\, v_i\in U_i\ \text{for every}\ i=0,1,2,3\}
$$
and we let $\Omega$ be the projection of $\widetilde{\Omega}$ in $\overline{S}_3^*(N)$. Of course, $\widetilde{\Omega}$ is an open neighbourhood
of $\widetilde{\Delta}_0$ in $\overline{S}_3^*({\matH}^3)$, and since the projection $\overline{S}_3^*(\matH^3)\to \overline{S}_3^*(N)$ is open,
the set $\Omega$ is an open neighbourhood of $\Delta_0$ in $\overline{S}_3^*(N)$. 

Let now $f\colon  \overline{S}_3^*(N)$ be any continuous
compactly supported function 
such that $f(\Delta_0)=1$. 
Recall that the vertices of the lifts of non-residual tetrahedra of $T'_i$ lie on the boundary of (deeper and deeper, as $i\to +\infty$) removed horoballs
centered at the ideal vertices of lifts of tetrahedra of $T$. 
We say that a simplex $\sigma'$ appearing in the cycle $c_i$ is a \emph{relative} of $\Delta_0$ if 
it is non-residual and it admits a lift to $\matH^3$ with vertices on horospheres centered
at the ideal vertices of a lift of $\Delta_0$ (in the correct order).

Thanks to our definition of $\Omega$,
we can choose $i\in \mathbb{N}$ such that, if $\sigma$ is a non-residual simplex appearing in $c_i$, then $\sigma$ belongs to $\Omega$ 
if and only if it is a relative of $\Delta_0$. Let us now decompose $c_i$ as follows:
$$
c_i=c_i^0+c_i^{nr}+c_i^r\ ,
$$
where $c_i^0$ is supported on relatives of $\Delta_0$, $c_i^{nr}$ is supported on non-residual simplices which are not relatives of $\Delta_0$, and
$c_i^r$ is supported on residual simplices. Since the simplices appearing in $c_i^{nr}$ cannot belong to $\Omega$ for $i$ sufficiently large, we have 
\begin{equation}\label{nr}
\lim_{i\to +\infty} \int_\Omega f\, d\Theta(c_i^{nr})=0\ .
\end{equation}
Recall now that the alternating operator associates to every simplex the average of 24 singular simplices, and that positively-oriented simplices come with the coefficient
$+1/24$. Therefore, $c_i^0$ is a linear combination of $d_i$ simplices, each of which comes with the real coefficient $1/(24d_i)$. In particular,
we have $\|c_i^0\|=1/24$. In the very same way, if one starts with a negatively oriented
$\Delta_0$, still $\|c_i^0\|=1/24$ but the coefficients appearing in $c_i^0$ are all negative. As a consequence, since $f(\Delta_0)=1$ and the   simplices appearing
in $c_i^0$ are converging to $\Delta_0$ in $\overline{S}^*_3(N)$ (and $f$ is continuous), we get
\begin{equation}\label{nr0}
\lim_{i\to +\infty} \int_\Omega f \, d\Theta(c_i^0)=\|c^0_i\|=\frac{1}{24}
\end{equation}
(while, if $\Delta_0$ were negatively oriented, we would have $\lim_{i\to +\infty} \int_\Omega f \, d\Theta(c^0_i)=-\|c^0_i\|=-\frac{1}{24}$). 

Finally from~\eqref{pochi} we deduce that $\lim_{i\to +\infty} \|c_i^r\|=0$, hence
\begin{equation}\label{r}
\lim_{i\to +\infty} \int_\Omega f\, d\Theta(c_i^{r})=0\ .
\end{equation}
 Putting together \eqref{nr}, \eqref{nr0} and \eqref{r} we then obtain
  $$
\lim_{i\to +\infty} \int_\Omega f \, d\Theta(c_i)=\pm\frac{1}{24}\ ,
$$
where the sign depends on the fact whether $\Delta_0$ is positively or negatively oriented. 

Let us now denote by $\mu$ the limit of $\Theta(c_i)$ (which we may assume to exist, up to passing to a subsequence; in fact, with a little more effort we could easily prove that the sequence $\Theta(c_i)$, $i\in\mathbb{N}$, is itself convergent).
Due to the definition of weak-* convergence, we have thus proved that
there exists a neighbourhood $\Omega$ of $\Delta_0$ such that, for every compactly supported $f\colon \overline{S}^*_3(N)\to\mathbb{R}$ with $f(\Delta_0)=1$, we have
$$
\int_\Omega f\, d\mu = \pm \frac{1}{24}\ .
$$
This implies that $\mu(\{\Delta_0\})=\pm 1/24$. 

We have thus shown that $\mu(\{\Delta_0\})=\pm 1/24$ for every tetrahedron $\Delta_0\in \reg(N)$ whose geometric realization is a tetrahedron of the ideal triangulation $T$
we started with. But every ideal tetrahedron of $T$ gives rise to 24 tetrahedra in $\reg(N)$, and the simplicial volume $\|N\|$ is equal to the number of tetrahedra of $T$,
hence the contribution to $\mu$ of the atomic measures supported by tetrahedra whose geometric realizations are in $T$ has total variation equal to $\|N\|$. 
Since we already know from Theorem~\ref{totalvariation} that $\|\mu\|=\|N\|$, this finally implies that
$\mu=\mu_T$, as desired.
\end{proof}

We can now conclude the proof of Theorems~\ref{characterization} and~\ref{summarize} by showing that, if $N$ is commensurable with the Gieseking manifold, then it admits
non-equidistributed efficient cycles. 

\subsection{Proof of Theorem~\ref{characterization}}
We have proved in Section~\ref{unique:sec} that, if $N$ is not commensurable with the Gieseking manifold, then every efficient cycle for $N$ is equidistributed.

Viceversa,  if $N$ is commensurable with the Gieseking manifold, then there exists a degree-$d$ covering $p\colon \hat{N}\to N$, where $\hat{N}$ admits a triangulation $\hat{T}$ by regular
ideal tetrahedra. Let $\hat{c}_i$, $i\in\mathbb{N}$, be the relative fundamental cycles for $\hat{N}$ constructed in the proof of Theorem~\ref{explicit:thm}, and for every 
$i\in\mathbb{N}$ let $c_i=p_*(c_i)/d$. The covering map $p$ induces a continuous map $\overline{S}^*_3(\hat{N})\to \overline{S}^*_3({N})$, hence a map 
$\calM(\overline{S}^*_3(\hat{N}))\to \calM(\overline{S}^*_3({N}))$. The very same proof of Theorem~\ref{explicit:thm} shows that the limit
$\mu=\lim_{i\to +\infty} \Theta(c_i) \in \calM(\overline{S}^*_3({N}))$ is an efficient cycle for $N$, and 
 is equal to the image of $\mu_{\hat{T}}$ via the map
$\calM(\overline{S}^*_3(\hat{N}))\to \calM(\overline{S}^*_3({N}))$. But the image of a purely atomic measure via a continuous map is itself purely atomic.
In particular, $\mu$ is a non-equidistributed efficient cycle for $N$, and this concludes the proof.

\subsection{Proof of Theorem~\ref{summarize}} We are only left to show that,
if $N$ is not commensurable with the Gieseking manifold and $c_i$, $i\in\mathbb{N}$ is any minimizing sequence for $N$, then
$$
\lim_{i\to +\infty} \Theta(c_i)=\frac{1}{2v_n} \mu_\eq\ .
$$
Of course, it is sufficient to show that every subsequence of $c_i$, $i\in\mathbb{N}$ admits a subsequence whose image via $\Theta$ converges to $\mu_\eq/(2v_n)$. 
However, the total variation of the measures $\Theta(c_i)$ is uniformly bounded, hence by compactness of the unit ball in $\calM(\overline{S}^*_n(N)$ every subsequence
of $\Theta(c_i)$ admits a subsequence converging to some measure $\mu\in \calM(\overline{S}^*_n(N)$. By Corollary~\ref{description} we must have
$\mu=\mu_\eq/(2v_n)$, and this concludes the proof.

\bibliography{biblio_efficient}
\bibliographystyle{alpha}

\end{document}